\documentclass[12pt,twoside]{amsart} \usepackage{amssymb,color,tikz}
\usetikzlibrary{matrix} \nonstopmode \textwidth=16.00cm

\usepackage{xcolor}
\usepackage{palatino}

\numberwithin{equation}{section} 

\newcommand{\CC}{\mathbb{C}}

\newcommand{\ZZ}{\mathbb{Z}} 
\newcommand {\PP}{\mathbb{P}}
\newcommand {\sI}{\mathcal{I}}
\newcommand {\sO}{\mathcal{O}}
 \DeclareMathOperator{\Proj}{Proj}
 \def\cocoa{{\hbox{\rm C\kern-.13em
      o\kern-.07em C\kern-.13em o\kern-.15em A}}}
\newtheorem{theorem}{Theorem}[section]
\newtheorem{lemma}[theorem]{Lemma}
\newtheorem{proposition}[theorem]{Proposition}

 \theoremstyle{definition}
\newtheorem{definition}[theorem]{Definition} \theoremstyle{remark}
\newtheorem{remark}[theorem]{Remark}
\newtheorem{example}[theorem]{Example}

\definecolor{MyDarkGreen}{cmyk}{0.7,0,1,0}

\begin{document}

\title[Configurations of eigenpoints]
{Configurations of eigenpoints}
 \author[V. Beorchia]{Valentina Beorchia} 
 \address{Dipartimento di Matematica e  Geoscienze, Universit\`a di
Trieste, Via Valerio 12/1, 34127 Trieste, Italy}
  \email{beorchia@units.it, 
  ORCID 0000-0003-3681-9045}.
  \author[R.\ M.\ Mir\'o-Roig]{Rosa M.\ Mir\'o-Roig} 
  \address{Facultat de
  Matem\`atiques i Inform\`atica, Universitat de Barcelona, Gran Via des les
  Corts Catalanes 585, 08007 Barcelona, Spain} \email{miro@ub.edu, ORCID 0000-0003-1375-6547}

\thanks{The first author is a member of GNSAGA of INdAM and is supported by the fund Universit\`a degli Studi di Trieste - FRA 2022, by the MIUR Excellence Department Project awarded to the DMG of Trieste, 2018-2023, and by the Italian MIUR funds, PRIN project Moduli Theory and Birational Classification (2017).} 
\thanks{The second author has been partially supported by the grant PID2019-104844GB-I00}


\begin{abstract} This note is motivated by the Question 16
of http://cubics.wikidot.com (see also \cite{RS}): Which configurations of 15 points in $\PP^3$ arise as eigenpoints of a cubic surface? We prove that a general eigenscheme in $\PP^n$ is the complete intersection of two suitable smooth determinantal curves on a smooth determinantal surface. Moreover, we prove that the converse result holds if $n=3$, providing an answer in any degree to the cited question. Finally, we show that any general set of points in $\PP^3$ can be enlarged to an eigenscheme of a partially symmetric tensor.
\end{abstract}

\maketitle

\section{Introduction}

Tensor spectral theory fits in a natural way in the context of Algebraic Geometry, and it has recently received a lot of  attention, also because of the extended range of applications, like hypergraphs theory, quantum physics, dynamical systems, polynomial optimization, signal processing and medical imaging among others (see, for instance, \cite{QCC}). In this article we focus on eigenvectors of tensors, which were introduced independently in \cite{Lim} and \cite{Qi}. The definition relies on the choice of a nondegenerate quadratic form $Q$ on $\CC^{n+1}$, or equivalently on the choice of an isomorphism $\CC^{n+1} \cong {\CC^{n+1}}^\vee$. With such a choice, given $T \in ({\CC^{n+1}} ^\vee)^{\otimes d}$, a non zero vector $v \in \CC^{n+1}$ is said to be an {\em eigenvector} for $T$ if the contraction $T(v ^{\otimes d-1})\in {\CC^{n+1}}^\vee$ is equal to $\lambda v$ for some $\lambda \in \CC$ under the isomorphism ${\CC^{n+1}}^\vee \cong \CC^{n+1}$.
Since the property of being an eigenvector
is preserved under scalar multiplication, it is natural to regard eigenvectors as points in $\PP^n$, in which case we shall call them {\em eigenpoints}. By choosing an orthonomal basis with respect to $Q$ and coordinates $T=(t_{i_1i_2\cdots i_d})$, the eigenpoint
condition can be expressed with the algebraic equations 
\begin{equation}\label{eq: general equations}
\sum _{i_2 \dots i_d =0}^n t_{j i_2 \dots i_d} x_{i_2} \dots x_{i_d} = \lambda x_j,\  {\text for} \ j =0 \dots n.
\end{equation}

The complete set of eigenpoints with the natural scheme structure arising from its equations will be called an eigenscheme. 

There are several natural problems related to eigenschemes: their dimension, their degree and their configurations. Concerning their dimension, Abo proved in
\cite{Abo} that the discriminant locus of tensors in $\PP(({\CC^{n+1}} ^\vee)^{\otimes d})$ with positive-dimensional or non-reduced eigenscheme is a hypersurface, and he computed its degree. So the general tensor has a $0$-dimensional eigenscheme, and its degree has been determined in \cite{CS}, see also Lemma \ref{lemma: prelim}.

Regarding the geometry of their configurations, the planar case has been settled in \cite{ASS} for $d=3$ and in \cite {BGV} for $d \ge 4$. It turns out that in the planar case the eigenschemes of general tensors are characterized geometrically by the property of
being the base locus of a net of degree $d$ curves, and of having no aligned subschemes of length $d+1$, nor length $kd$ subschemes on a curve of degree $k$ for any $2 \le k \le d-1$.

Finally, in the case $n=3=d$, some partial results are given in \cite{CGKS}, and for $n \ge 3$, a characterization of the sets of determinantal equations is given in \cite {BGV2022}. As far as we know, a general description of the geometry of eigenpoints in the projective space is widely open.

This note is motivated by Question 16
of the website http://cubics.wikidot.com (see also \cite{RS}): Which configurations of 15 points in $\PP^3$ arise as eigenpoints of a cubic surface?
We give a geometric descriptions of eigenpoints configurations in $\PP^n$. Our main results are the following: a general eigenscheme in $\PP^n$ is the complete intersection of two smooth determinantal curves on a smooth determinantal surface, see Theorem \ref{geo_descr}. We prove that the converse result holds if $n=3$ (see Theorem \ref{converse}), providing an answer to the cited question. Finally, in Proposition \ref{degw}, we show that any general set of points in $\PP^3$ can be enlarged to an eigenscheme of a partially symmetric tensor.

Our approach relies on the fact that eigenschemes of general tensors turn out to be standard determinantal subschemes, see Definition \ref{def: eigenscheme}. By deleting any column of the defining matrix, the maximal minors of the resulting matrix define a smooth irreducible standard determinantal curve, and deleting two columns we similarly find a smooth standard determinantal surface. Such a surface is not general, since it contains a line. By using the results of \cite{KMR} and \cite{G} it is possible to determine the class of the curves on the surface, and to obtain the first result.

The proof of the converse in $\PP^3$ relies on results of \cite{L} on the Picard group of a general surface containing a line and on the use of the Mapping cone process.

Finally, some similar ideas allow to proof
Proposition \ref{degw}.

\vskip 2mm

\noindent {\bf Acknowledgement}. Most of this work was done while the second author was a guest of the University of
Trieste, and she would like to thank the people of  the Dipartimento di Matematica e Geoscienze for their warm hospitality.

\vskip 2mm

\noindent {\bf Notation. }Throughout this note,  $R=\CC[x_0,x_1,\cdots , x_n]$ and $\PP^n=\Proj(R)$. Given a homogeneous matrix $ M$, i.e. a matrix representing a degree 0 morphism $\Phi $
of free graded $R$-modules, we denote by $I_t(M)$  the ideal of $R$ generated by the $t\times t$
 minors of $M$.

\section{Preliminaries}

This section contains the basic definitions and results on eigenschemes and it lays the groundwork for the results in the later sections.

As observed in \cite[Section 1]{ASS}, the eigenscheme equations \eqref{eq: general equations} involve $n+1$ homogeneous polynomials of degree $d-1$, so that it is not restrictive to assume that $T$ is a
{\it partially symmetric tensors}, that is for any fixed $j\in \{0, \dots, n \}$, the {\it slice} $(t_{j i_2 \dots i_d})$ is a symmetric tensor of order $d-1$.
Indeed, it is clear from \eqref{eq: general equations} that any set of eigenvectors can be realized as the set of eigenvectors of a partially symmetric tensor $T\in (Sym^{d-1}\CC^{n+1}) \otimes \CC^{n+1}$. Finally, we can identify $T$ with a $(n+1)$-uple $(g_0,\cdots ,g_n)$ of homogeneous polynomials of degree $d-1$. With such identifications, we immediately see that the equations  \eqref{eq: general equations} are of determinantal type, and we are in position to give the explicit definition of eigenscheme. We first recall the following notion.

\begin{definition} A codimension $c$ subscheme $X\subset \PP^n$ is called a {\em standard determinantal
scheme} if the saturated homogeneous ideal $I(X) = I_t(A)$ for some $t\times (t+c-1)$ homogeneous matrix $A$.
\end{definition}

\begin{definition} \label{def: eigenscheme}
Let $T=(g_0,\cdots ,g_n)\in (Sym^{d-1}\CC^{n+1})^{\oplus (n+1)}$ be a partially symmetric tensor. The eigenscheme of $T$ is the closed subscheme $E(T)\subset \PP^n$ defined by the $2\times 2$ minors of
\begin{equation}\label{def_matrix}
M=\begin{pmatrix} x_0 & x_1 & \cdots & x_n \\
g_0 & g_1 & \cdots & g_n
\end{pmatrix};
\end{equation}
i.e. $I(E(T))=I_2(M)$. If $T$ is general, then $E(T)$ is $0$-dimensional (see, for instance, \cite{Abo}).
Moreover, by the Hochster-Eagon Theorem \cite{HE}, the quotient $R/I(E(T))$ is a Cohen-Macaulay ring, and as a consequence, $I(E(T))$ is saturated. Hence $E(T)$ is a standard determinantal scheme.
When $T$ is symmetric, i.e., there is a homogeneous polynomial $f$ and $g_i=\partial _if$ we denote its eigenscheme by $E(f)$.
\end{definition}

The eigenschemes arise also in the following framework:  if we consider the rational map 
$$
\mu:\PP^n \dashrightarrow \PP^n, \quad p\mapsto \mu (p)=(g_0(p),g_1(p),\cdots ,g_n(p)),
$$ 
then $p\in \PP^n$ is an eigenpoint if and only if either $p$ is a fixed point of $\mu $ or $\mu $ is not defined at $p$. Therefore, in the symmetric case, the eigenscheme $E(f)$ consists of the Jacobian scheme and the fixed points of the polar map $\nabla f=(\partial_0f, \cdots ,\partial _nf)$. Basic questions concerning the degree, the dimension and the minimal free resolution of eigenschemes $E(T)$ of a general symmetric tensor  $T=(g_0,\cdots ,g_n)\in (Sym^{d-1}\CC^{n+1})^{\oplus (n+1)}$ are well known and we recall them in next proposition for seek of completeness. In next section, we will use the geometry of arithmetically Cohen-Macaulay curves and surfaces naturally associated with eigenschemes to gain some knowledge about the configuration of eigenschemes.

\begin{lemma} 
\label{lemma: prelim}
Fix integers $n\ge 2$ and $d\ge 2$.
 Let $T=(g_0,\cdots ,g_n)\in (Sym^{d-1}\CC^{n+1})^{\oplus (n+1)}$ be a general partially symmetric tensor. It holds:
 \begin{itemize}
     \item[(1)] $E(T)$ is a reduced 0-dimensional scheme of length 
     $$
     \frac{(d-1)^{n+1}-1}{d-2}.
     $$
     \item[(2)] The homogeneous ideal $I(E(T))\subset R$ has a minimal free $R$-resolution
     $$ 0 \longrightarrow \oplus _{i=0}^{n-1}R(-(i+1)d-n+2i+1) \longrightarrow   \cdots \longrightarrow $$
     $$R(-1-d)^{\binom {n+1} {3}}\oplus R(1-2d)^{\binom {n+1}{3}}
     \longrightarrow R(-d)^{\binom {n+1}{2}}\longrightarrow I(E(T)) \longrightarrow 0.
     $$
     \end{itemize}
\end{lemma}
\begin{proof}
(1) See \cite{CS}.

(2) Since $E(T)$ is a standard determinantal scheme of codimension $n$, its minimal free resolution is given by the Eagon-Northcott complex (see \cite[Theorem A2.10]{E}).
\end{proof}

It is worthwhile to point out that for any 
partially symmetric tensor $T=(g_0,\cdots ,g_n)\in (Sym^{d-1}\CC^{n+1})^{\oplus (n+1)}$ the associated eigenscheme $E(T)$ contains no $d+1$ points on a line and, more general, no $sd+1$ points on a curve of degree $s$. Indeed,  $E(T)$
is a 0-dimensional scheme whose homogeneous ideal is generated by forms of degree $d$. If $E(T)$ contains $sd+1$ points on a curve $C$ of degree $s$ then, by Bezout's theorem, the whole curve $C$ will be contained in the base locus contradicting the fact that $E(T)$ is 0-dimensional.

\vskip 2mm We will end this preliminary section with a concrete example.

\begin{example} \label{ex: Fermat}
\rm
We consider the Fermat cubic $f=x_0^3+x_1^3+x_2^3+x_3^3\in \CC [x_0,x_1,x_2,x_2]$. The eigenscheme $E(f)$ is the 0-dimensional subscheme of $\PP^3$ of length 15 defined by the maximal minors of 
$$
M=\begin{pmatrix} x_0 & x_1 & x_2 & x_3 \\
x^2_0 & x^2_1 & x_2^2 & x_3^2
\end{pmatrix}.
$$
Therefore, $E(f)=\{(1,0,0,0),(0,1,0,0),(0,0,1,0),(0,0,0,1),(1,1,0,0),(1,0,1,0),(1,0,0,1),  \\
(0,1,1,0),(0,1,0,1),(0,0,1,1),(1,1,1,0),(1,1,0,1),(1,0,1,1),(0,1,1,1),(1,1,1,1)\}$
\end{example}

\section{Eigenschemes of partially symmetric tensors}
In this section, we deal with partially symmetric tensors and we address the open problem of determining all possible subschemes of $\PP^n$ arising as their eigenschemes. We will be able to give for any $n\ge 3$ a necessary condition which turns out to be sufficient when $n=3$.
\begin{remark}\label{curves_and_surfaces}
Let $Z\subset \PP^n$ with $n \ge 3$ be an eigenscheme of a general partially symmetric tensor $T=(g_0,\cdots ,g_n)\in (Sym^{d-1}\CC^{n+1})^{\oplus (n+1)}$, with defining matrix $M$ as in \eqref{def_matrix}.

Moreover, denote by
$$
M_i:=\begin{pmatrix} x_0 & x_1 & \cdots & \widehat {x_i} & \cdots & x_n \\
g_0 & g_1 & \cdots & \widehat {g_i} & \cdots& g_n
\end{pmatrix}
$$ 
the matrix obtained from $M$ by deleting the $i$-th column.

Then the $2\times 2$ minors of $M_i$ define a standard aCM determinantal curve $C_i\subset \PP^n$, which is smooth and irreducible if the forms $g_j$ are general.

We claim that this holds also in the symmetric case, where $g_i = \partial_i f$, if $f$ is a sufficiently general homogeneous polynomial. 

Indeed, first observe that for a general $f$, the
determinantal subscheme $C_i$ defined by the $2\times 2$ minors of
$$
\begin{pmatrix} x_0 & x_1 & \cdots & \widehat {x_i} & \cdots & x_n \\
\partial_0 f& \partial_1 f& \cdots & \widehat {\partial_i f} & \cdots& \partial_n f
\end{pmatrix}
$$
has pure dimension $1$, hence it is of the maximal codimension $n-1$. To verify this statement, note that such a property is satisfied by an open subscheme $U$ of the space of all partially symmetric tensors. The symmetric tensors form an irreducible closed subvariety, so in order to prove the nonemptyness of the intersection of $U$ with the symmetric locus, it suffices to exhibit one such example. This is given, for instance, by the Fermat polynomial
$$
f=x_0^d + \dots + x_n^d,
$$
where the defining matrix of the eigenscheme is of the type of the Example 
\ref{ex: Fermat}. It is not difficult to see that by deleting a column, the 
locus $C_i$ a set of lines, all passing through one of the
fundamental points.

Next we observe that, if $f$ is general, we have that $C_i$
is smooth. Indeed, the singular locus is the common zero locus of the matrix entries,
and this is, in general, empty.

As a consequence, the ideal sheaf of such a subscheme is resolved by the Eagon-Northcott complex and, hence,  $C_i$ is an aCM subscheme. In particular, we have $H^1(\sI_{C_i})=0$. Therefore $H^0(\sO_{C_i})=\CC$. This implies that $C_i$ is connected and being smooth we get that $C_i$ is irreducible.

Now we turn again to the general case of partially symmetric tensors, and we determine the degree of $C_i$.
Since the section of $C_i$ with the hyperplane $x_i=0$ is, by construction, a general 0-dimensional eigenscheme in 
$\PP^{n-1}$, its degree is (Lemma \ref{lemma: prelim}(1))
$$
{\rm deg} \ C_i=\frac{(d-1)^n -1}{ d-2}.
$$
Similarly, for any pair of indices $i \neq j$, we can consider the matrix
$$
M_{ij}:=\begin{pmatrix} x_0 & x_1 & \cdots & \widehat {x_i} & \cdots &  \widehat {x_j} & \cdots &  x_n \\
g_0 & g_1 & \cdots & \widehat {g_i} & \cdots& \widehat {g_j} & \cdots & g_n\\
\end{pmatrix}.
$$ 
Then the $2\times 2$ minors of $M_{ij}$ define a standard determinantal surface $S_{ij}\subset \PP^n$, which is smooth and irreducible if the forms $g_k$ are general.

Again, by similar arguments, we see that this holds also in the symmetric case, if $f$ is a sufficiently general homogeneous polynomial.

Since the section of $S_{ij}$ with the codimension $2$ linear subspace $x_i=x_j=0$ is by construction a general eigenscheme in 
$\PP^{n-2}$, its degree is  (Lemma \ref{lemma: prelim}(1))
$$
{\rm deg} \ S_{ij}=\frac{(d-1)^{n-1} -1} {d-2}.
$$

\end{remark}

We are now in the position to state and prove our first result.
\begin{theorem}\label{geo_descr}
Let $Z\subset \PP^n$ with $n \ge 3$ be an eigenscheme of a general partially symmetric, or general symmetric, tensor of order $d \ge 3$, and let $i\neq j$, $i,j \in \{ 0,\dots,n\}$ be arbitrary.

Then $Z$ is the complete intersection of the two aCM curves
$$
Z=C_i \cap C_j
$$
on the smooth surface $S_{ij}$.
\end{theorem}

\begin{proof}
Without loss of generality we can assume, for simplicity, that $i=0$ and $j=1$.

We assume first that
$$
n\ge 4 \text{ or } d \ge 4.
$$
By denoting with $H$ a hyperplane divisor of $S_{01}$ and with $L$ the line given by the equations
$$
L: \quad x_2=x_3= \cdots = x_n=0,
$$
by \cite [Theorem 4.1]{KMR} we have
$$
{\rm Pic} (S_{01}) \cong \ZZ^2 \cong \langle H, L \rangle.
$$
To determine the class of $C_i$ as a divisor, we 
recall the well known fact that ${\mathcal I}_{C_i/S_{01}}\cong{\mathcal I}_{L/S_{01}}(\lambda ) $ for some $\lambda \in\ZZ$, i.e. $C_i \sim L+\lambda H$ for some $\lambda \in \ZZ$ (see \cite{G}). By considering
 the relation
 $$
 C_i \cdot H= \frac{(d-1)^n -1}{d-2},
 $$
we get that 
 $$
\lambda = d-1.
$$

Next we determine the self-intersection $L^2$ of the line $L$. Observe that by 
\cite[Proposition 2.2 (iii)]{KMR} we have that the canonical class 
$$
K_{S_{01}} \sim (n-3)C+(d-n-1)H\sim (n-3)L+((n-3)(d-1)+d-n-1))H.
$$
Then by the adjunction formula we have
$$
2g(L)-2 =-2 = (K_{S_{01}} +L) \cdot L = ((nd-2d-2n+2)H+(n-2)L)) \cdot 
$$
This gives 
$$
(n-2)L^2=(n-2)(d-2), \ i.e. 
$$
$$
L^2 = 2-d.
$$
As a consequence we have
$$
C_0 \cdot C_1 = ((d-1)H + L)^2 = \frac{(d-1)^{n+1} -1}{d-2}= {\rm deg} \   (Z).
$$
Therefore $Z$ is the complete intersection of $C_0$ and $C_1$.

Finally, let us treat the case $n=d=3$. In this case, $S_{01}$ is the blow up of $\PP^2$ at 6 points and $
{\rm Pic} (S_{01}) \cong \ZZ^7 \cong \langle \ell, e_1,e_2,e_3,e_4,e_5,e_6 \rangle.
$ where $\ell$ is the pullback of a line in $\PP^2$ and $e_{i}$ are the exceptional divisors. The line $L: \quad x_2=x_3=0$ of $S_{01}$ can be identified with one of the following divisors: $e_j$, $\ell -e_{j_1}-e_{j_2}$, $2\ell -e_{j_1}-e_{j_2}-e_{j_3}-e_{j_4}-e_{j_5}$ and $H\sim 3\ell-e_1-e_2-e_3-e_4-e_5-e_6$. Applying again \cite{G} we have 
 ${\mathcal I}_{C_i/S_{01}}\cong{\mathcal I}_{L/S_{01}}(\lambda ) $ for some $\lambda \in\ZZ$, i.e. $C_i \sim L+\lambda H$ for some $\lambda \in \ZZ$ and using the fact that $C_i \cdot H =7$ we get that $\lambda =2$. The remaining part of the proof works as in the case $n\ge 4$ or $d\ge 4$.
\end{proof}

In what follows we shall see that the converse holds if $n=3$ which for the particular case $d=3$ will provide an answer to  Question 16
of the website http://cubics.wikidot.com: Which configurations of 15 points in $\PP^3$ arise as eigenpoints of a cubic surface?

\vskip 2mm
The following result clarifies the necessity of the assumptions of Theorem \ref{converse}.

\begin{proposition}
 Let $Z =E(T) \subset \PP^n$ be a $0$ - dimensional general reduced eigenscheme. Then no $(d-1)\frac{(d-1)^n-1}{d-2}$ points of $Z$ lie on a hypersurface of degree $d-1$. 
\end{proposition}

\begin{proof}
Let \eqref{def_matrix} be the matrix associated with $Z$. Assume by contradiction that $Z$ admits a subscheme $Z' \subset Z$ of $(d-1)\frac{(d-1)^n-1}{d-2}$ points lying on a hypersurface 
$$
\Sigma =V(h) \subseteq \PP^n
$$
of degree $d-1$.

By using the notations of Remark \ref{curves_and_surfaces} and as in the proof of Theorem \ref{geo_descr} we denote by $H$ a hyperplane divisor of $S_{01}$ and by $L$ the line given by the equations
$$
L: \quad x_2=x_3= \cdots = x_n=0.
$$
We
set $D' = S_{01} \cap \Sigma \sim (d-1)H$. As $Z' \subset Z$ and $Z$ is a complete intersection
of two irreducible divisors $C_0, C_1 \in | (d-1)H +L|$
by Theorem
\ref{geo_descr}, we have that
$$
Z' \subseteq D' \cap C_i.
$$
Since we have
$$
D' \cdot C_i = (d-1)H \cdot C_i= (d-1) \frac {(d-1)^n-1}{ d-2} = {\rm deg} \   Z',
$$
we see that $Z' = D' \cap C_i$. Moreover, observe that residually to $Z'$ in $Z$ we get one point $P$.

Now we claim that $P \in L$. Indeed, since both $Z$ and $Z'$ are complete intersection schemes in $S_{01}$, the minimal resolutions of their  $\sI_{Z,S_{01}}$ and $\sI_{Z',S_{01}}$ in $S_{01}$ are given by the Koszul complex. More precisely, we have
$$
0 \to \sO_{S_{01}} (-2(d-1)H-2L) \to 2\sO_{S_{01}}(-(d-1)H -L) \to\sI_{Z,S_{01}} \to 0, \text{ and }
$$
$$
0 \to \sO_{S_{01}} (-2(d-1)H-L) \to \sO_{S_{01}}(-(d-1)H -L) \oplus \sO_{S_{01}}(-(d-1)H)\to\sI_{Z',S_{01}} \to 0.
$$
Since $Z'$ and $P$ are directly linked by $Z$, by applying the Mapping cone process, we get a resolution of the ideal sheaf
$\sI_{P,S_{01}}$:
$$
0 \to \sO_{S_{01}} (-(d-1)H-L) \oplus \sO_{S_{01}} (-(d-1)H-2L)\to \qquad \qquad \qquad \qquad
$$
$$
\qquad \qquad \qquad \qquad \to 2\sO_{S_{01}}(-(d-1)H -L) \oplus \sO_{S_{01}}(-L)\to\sI_{P,S_{01}} \to 0,
$$
which allows us to conclude that $P\in L$.

As a consequence, the divisor $D' + L$ contains $Z$. So, it belongs to the pencil 
$$
D' +L \in \langle C_0, C_1\rangle$$
whose base locus is $Z$. Therefore $Z$ is a complete intersection also of $C_1$ and the reducible curve $D' \cup L$.

By construction, the equations in $S_{01}$ of $D' \cup L$
are 
$$
x_i h =0, \  {\rm for} \  i=2, \dots, n+1
$$
hence the saturated homogeneous ideal of $D' \cup L$ in $\PP^n$ is generated by the $2 \times 2$ minors of the matrix
$$
\begin{pmatrix} x_2 & x_3 & \cdots &  x_n & 0 \\
g_2 & g_3 & \cdots & g_n &h
\end{pmatrix}.
$$ 
Finally, we set
$$
N:=\begin{pmatrix} x_1 & x_2 & x_3 & \cdots &  x_n & 0 \\
g_1 & g_2 & g_3 & \cdots & g_n &h
\end{pmatrix},
$$ 
and we consider $W \subset \PP^n$ the determinantal scheme defined by the $2 \times 2$ minors of $N$:
$$
I(W)= I_2 (N).
$$
By construction we have that $W$ is $0$-dimensional and
$
Z \subseteq W
$.
As ${\rm deg} \   (Z)={\rm deg} \   (W)$, the equality $Z=W$ holds.

So $I(Z) = I_2(M)=I_2(N)$, but this is a contradiction, since any defining matrix for an eigenscheme has the property that the linear entries of the first row are linearly independent (see, for instance, \cite[Corollary 2.9]{BGV2022}).

\end{proof}

Taking into account all the necessary conditions that we have proved so far, we can state the converse result for points in $\PP^3$.

\begin{theorem}\label{converse}
Let $Z \subset \PP^3$ be a reduced $0$ - dimensional subscheme of degree
$$
{\rm deg} \  (Z) = \frac {(d-1)^4 -1}{ d-2}= d(d^2-2d+2)
$$
for some $d \ge 3$.

If $Z$ is a complete intersection of two irreducible aCM curves $C_0$ and $C_1$ on a smooth general surface $S$ of degree $d$ containing a line $L$, and the following conditions are satisfied:
\begin{enumerate}\label{assumptions}
\item no $(d-1)(d^2-d+1)$ points of $Z$ lie on a surface of degree $d-1$;
    \item $C_0 \sim C_1$;
    \item ${\rm deg} \     C_0 = {\rm deg} \  C_1 = \frac {(d-1)^3 -1}{ d-2} =d^2-d+1$;
    \item $g(C_0)=g(C_1)= d^3 - 7 \frac {d(d-1)}{ 2} -1$,
\end{enumerate}
then $Z$ is the eigenscheme of some partially symmetric tensor of order $d$.
\end{theorem}

\begin{proof}
With no loss of generality we may assume that $L$ has equations
$$
L: \quad x_0 =x_1 =0.
$$
Since $L\subset S$, the equation of $S$ can be written in the form
$$
S: \quad x_0 g_1 - x_1 g_0=0,
$$
for some suitable degree $d-1$ forms $g_0$ and $g_1$. 

Our next goal is to show that both $C_0$ and $C_1$ are determinantal curves, whose equations are the $2 \times 2$ minors of a matrix given by adding a suitable column 
to
$$
\begin{pmatrix} x_0 & x_1 \\
g_0 & g_1 \\
\end{pmatrix}.
$$
We start by determining the divisor class of $C_i$ in $S$.

If $d \ge 4$, by \cite{L} we have that $Pic(S) \cong \ZZ^2 \cong \langle H, L\rangle$, where $H$ is a hyperplane divisor on $S$. So $C_i \sim \alpha H +\beta L$ for some integers $\alpha$ and $\beta$. By recalling $K_S \sim (d-4)H$, the relations
$$
C_i \cdot H = d^2-d+1, \quad 2g(C_i)-2= ((d-4)H+C_i)\cdot C_i
$$
yield the equations
$$
d\alpha + \beta =d^2-d+1, 
$$
$$
2d^3 - 7 d(d-1) -4= d\alpha^2 + (2-d)\beta^2 +2 \alpha \beta +(d-4)(\alpha d +\beta),
$$
 where we used the fact that $L^2 =2-d$ Indeed, $-2=2g(L)-2=(K_s+L)\cdot L=((d-4)H+L)\cdot L=d-4+L^2$). The two possible solutions are
 \begin{equation}\label{first}
 \alpha = d-1,\quad \beta =1,
 \end{equation}
 and 
 \begin{equation}\label{second}
 \alpha =\frac {(d^2-d+2)} {d},\quad \beta =-1.
\end{equation}
 As for $d \ge 3$ the solution \eqref{second} is not integer, it holds \eqref{first}. Therefore, we have
 $$
 C_i \sim (d-1)H +L.
 $$
 If $d=3$, by \cite[Theorem 4.7]{PT}, any aCM curve $C$ of degree $7$ and genus $5$ on a smooth cubic surface $S$ has class either
 $$
 C \sim e_i + 2H, \quad {\rm or} \ C \sim (l-e_j - e_k) +2H, \quad {\rm or} \ C \sim
 2l - e_{i_1}-e_{i_2}-e_{i_3}-e_{i_4}-e_{i_5}+2H,
 $$
 where the $e_i$'s are the exceptional divisors of $S$, when we identify it with the blow up of $\PP^2$ in $6$ points, and $l$ is the pull back of a line of $\PP^2$. We observe that the divisors
 $$
 e_i, \ l-e_j - e_k, \ 2l - e_{i_1}-e_{i_2}-e_{i_3}-e_{i_4}-e_{i_5}
 $$
 give rise to the $27$ lines of a smooth cubic surface, and by the assumption $C_0 \sim C_1$ we finally get that
 $$
 C_0 \sim C_1 \sim L + 2H,
 $$
 for some line $L \subset S$.
 
 In particular, in all cases we get
\begin{equation}\label{ideal}
 \sI_{C_i, S}= \sO_S(-C_i)=\sO_S (-(d-1)H -L)= \sI_{L,S}(-(d-1)H),
\end{equation}
 so we are led to determine a resolution of $\sI_{L,S}(-(d-1)H)$.
 
 By applying the Mapping Cone process (see, for instance, \cite[Section 1.5]{W}) to the two exact sequences
 $$
 \begin{array}{r}
 0 \to \sO_{\PP^3}(-d) \to \sI_{S,\PP^3} \to 0\\
 \\
 0 \to \sO_{\PP^3}(-2) \to 2\sO_{\PP^3}(-1) \to \sI_{L,\PP^3} \to 0\\
 \end{array}
 $$
 we get the resolution
 $$
 0 \to \sO_{\PP^3}(-2)\oplus \sO_{\PP^3}(-d) 
 \xrightarrow
{\left(
\begin{array}{cc}
x_0 & g_0\\
x_1& g_1\\
\end{array}
\right)
}
2\sO_{\PP^3}(-1)  \to \sI_{L,S} \to 0.
 $$
By \eqref{ideal} we have 
$
\sI_{C_i, S}\cong  \sI_{L,S}(-(d-1)H)
$
and we get, as a consequence, the resolution of $\sI_{C_i,S}$:
$$
0 \to \sO_{\PP^3}(-d-1)\oplus \sO_{\PP^3}(-2d+1) \to 2\sO_{\PP^3}(-d)  \to \sI_{C_i,S} \to 0.
$$
 Finally, we apply the Horseshoe Lemma to the diagram
 $$
 \begin{array}{rcl} 
 & 0 &  \\
 & \downarrow &  \\
 0 \to \sO_{\PP^3}(-d) \to & \sI_{S,\PP^3} & \to 0 \\
& \downarrow &  \\
&  \sI_{C_i,\PP3}&\\
& \downarrow &  \\
0 \to \sO_{\PP^3}(-d-1)\oplus \sO_{\PP^3}(-2d+1) \to 2\sO_{\PP^3}(-d)  \to & \sI_{C_i,S} & \to 0 \\
& \downarrow &  \\
& 0 &  \\
\end{array}
 $$
 and we obtain
 $$
 0 \to \sO_{\PP^3}(-d-1) \oplus \sO_{\PP^3}(-2d+1)
 \xrightarrow
{M_i:=\left(
\begin{array}{cc}
x_0 & g_0\\
x_1& g_1\\
l_i & f_i\\
\end{array}
\right)
}
3\sO_{\PP^3}(-d)  \to  \sI_{C_i,\PP^3}  \to 0,
 $$
 for suitable linear forms $l_i$ and suitable degree $d-1$ forms $f_i$. In particular, the saturated homogeneous ideal $I_{C_i,\PP^3}$ is generated in degree $d$ by the three $2\times 2$ minors of the matrix 
 above.
 
 Now we claim that the linear forms $x_0, x_1, l_i$ are linearly independent in both cases $i=0$ and $i=1$. Indeed, if this were not the case, by a base change we could transform the matrix $M_i$ in the matrix
 $$
 \left(
\begin{array}{cc}
x_0 & g_0\\
x_1& g_1\\
0 & h_i\\
\end{array}
\right),
 $$
 so the curve $C_i$ would be contained in the zero locus
 $$
 C_i \subset V(x_0 h_i, x_1h_i).
 $$
 Since $C_i$ is irreducible by assumption, we necessarily have
 $$
 C_i \subset V(h_i).
 $$
 As $C_i$ is also contained in the degree $d$ surface $S$,
by B\'ezout Theorem we would have
$$
{\rm deg} \  C_i \le (d-1)d,
$$
 and this contradicts the assumption (iii).

 Finally, we set
 \begin{equation}\label{matrix_oftwocurves}
 M:=\left(
\begin{array}{cc}
l_0 & f_0\\
x_0 & g_0\\
x_1& g_1\\
l_1 & f_1\\
\end{array}
\right),
\end{equation}
and set $W \subset \PP^3$ to be the determinantal scheme defined by the $2 \times 2$ minors of $M$:
$$
I(W)= I_2 (M).
$$
Since $Z$ is the complete intersection of the two determinantal curves $C_0$ and $C_1$, we have
$$
Z \subseteq W.
$$
As ${\rm deg} \   (Z)={\rm deg} \   (W)$, the equality $Z=W$ holds.

Finally, we claim that the four linear forms
 $$
 x_0, x_1, l_0, l_1
 $$
 are linearly independent.

 Assume by contradiction that
 this is not the case; then possibly performing a base change, we have that $Z$ is the degeneracy locus also of the matrix
 \begin{equation}\label{wrongmatrix_oftwocurves}
 M':=\left(
\begin{array}{cc}
l_0 & f_0\\
x_0 & g_0\\
x_1& g_1\\
0 & h_1\\
\end{array}
\right),
\end{equation}
 so we would have $Z \supseteq V(l_0 h_1, x_0h_1, x_1 h_1)$, and $Z$ would contain $(d-1)(d^2-d+1)$ points on the degree $d-1$ surface $ V(h_1)$, contradicting assumption (1).

To conclude, we observe that possibly performing a base change we may assume that $M$ is of the type
$$
\left(
\begin{array}{cc}
x_0 & g_0\\
x_1& g_1\\
x_2 & g_2\\
x_3 & g_3
\end{array}
\right),
$$
so $Z$ is indeed an eigenscheme.
\end{proof}

\begin{remark}
Theorem \ref{converse} gives a geometric criterion for a set of points $Z\subset \PP^3$ to be an eigenscheme of some partially symmetric tensor $T$, and our construction furnishes explicitly a set of generators for its ideal. Then by applying the criterion given in
\cite[Theorem 3.4, and Algorithm 2]{BGV2022}
it is possible to determine whether $Z$ is the eigenscheme of a symmetric tensor.
\end{remark}
We conclude this note by proving that any general set of points in $\PP^3$ can be enlarged to an eigenscheme of a partially symmetric tensor.

\begin{proposition}
Let $W \subset \PP^3$ be a set of $k$ general points and set $d\ge 3$ to be an integer such that
\begin{equation}\label{degw}
{\rm deg} \  W =k \le \binom {d-1}{3} + 3 \binom {d }{ 2} +1.
\end{equation}
Then $W$ can be embedded into an eigenscheme $Z=E(T)$ of a partially symmetric tensor of order $d$.
\end{proposition}

\begin{proof}
We observe that if $W$ is general enough, there always exists a smooth surface $S \subset \PP^3$ of degree $d$, containing a line $L$ and $W$. Indeed, we have
$h^0 (\sO_{\PP^3} (d))= \binom {d+3} {3}$, and the condition of containing a fixed line $L$ imposes $d+1$ linear conditions. By \eqref{degw} we have
$$
\binom {d+3 } {3} - (d+1) - {\rm deg} \  W \ge \frac {d(d+1)}{ 2},
$$
so the claim follows.

Without loss of generality we may assume that $L$ has equations $x_0=x_1=0$, so that $S$ is given by
$$
S : \ \ x_0 g_1 - x_1 g_0 =0,
$$
for some degree $d-1$ polynomials $g_0$ and $g_1$.

By \cite{L} the Picard group of such a surface
is
$$
{\rm Pic } (S) \cong \ZZ^2 \cong \langle H, L\rangle,
$$
where $H$ is a hyperplane divisor. Consider now the linear system
$|(d-1) H +L|$.
By Riemann Roch we have
$$
\chi (\sO_S((d-1)H+L)= \frac {1}{ 2} ((d-1)H+L) \cdot 
((d-1)H+L-K_S))+1 + p_a(S);
$$
recalling that $K_S \sim (d-4)H$ and $p_a(S)= \binom {d-1} {3}$, we obtain
$$
\chi (\sO_S((d-1)H+L)= \frac {1}{ 2} ((d-1)H+L) \cdot 
(3H+L)+1 + \binom {d-1 } {3}=3 \binom {d } {2} +3 +\binom {d-1 } { 3}.
$$
Moreover, since 
$$
(d-1)H+L \sim K_S +3 H +L
$$
and $3H+L$ is an ample divisor, by the Kodaira Vanishing Theorem we have \\ $h^1(\sO_S((d-1)H+L)=h^2(\sO_S((d-1)H+L)=0$,
so that
$$
\chi(\sO_S((d-1)H+L)=h^0(\sO_S((d-1)H+L).
$$

By the assumption \eqref{degw} there exists a pencil 
$\Lambda= \langle C_0, C_1 \rangle$ of curves in $|(d-1) H +L|$ containing $W$. Furthermore, by the generality of $W$ we may assume that no curve in $\Lambda$ has $L$ as an irreducible component, and that the base locus $Z$ of $\Lambda$ is zero - dimensional.

Then we have the isomorphisms given in \eqref{ideal} and we may apply the same argument of the Theorem \ref{converse} to show that the complete intersection subscheme $C_0 \cap C_1$ is the degeneracy locus of a matrix of the type
$\eqref{matrix_oftwocurves}$,
where the triples of linear forms $x_0, x_1, l_0$ and $x_0, x_1, l_1$ are linearly independent. We claim that also 
$x_0,x_1,l_0,l_1$ are linearly independent.

Assume by contradiction that
 this is not the case; then possibly performing a base change, we have that $Z$ is the degeneracy locus also of the matrix
 \eqref{wrongmatrix_oftwocurves},
and we see that there will be a curve $C$ in $\Lambda$ with equations $V(x_0 g_1- x_1 g_0, x_0 h_1, x_1 h_1)$. Since 
$C \supset L$, this contradicts our choice of $\Lambda$.

Hence we finally get that $Z$ is the degeneracy locus of a matrix
$$
\left(
\begin{array}{cc}
x_2 & g_2\\
x_0 & g_0\\
x_1& g_1\\
x_3 & g_3\\
\end{array}
\right),
$$
so it is an eigenscheme and $Z \supset W$ by construction.

\end{proof}


\begin{thebibliography}{ll}

\bibitem{Abo} {H. Abo},
	{\em On the discriminant locus of a rank $n-1$ vector bundle on
	$\PP^{n-1}$},
	{Portugaliae Mathematica},
	{\bf 77}, n. 3-4
	(2020), {299--343}.
	
	
	\bibitem{ASS} {H. Abo, A. Seigal and B. Sturmfels},
     {\em Eigenconfigurations of Tensors},
 {Algebraic and Geometric Methods in Discrete Mathematics, Contemporary Mathematics 685},
AMS, Providence, RI,
      (2017), 1--25.
	
\bibitem{BGV} V. Beorchia, F. Galuppi and L. Venturello, {\em Eigenscheme of Ternary Tensors}, SIAM J. Appl. Algebra Geometry {\bf 5} (2021), 620 - 650.

 \bibitem{BGV2022} V. Beorchia, F. Galuppi and L. Venturello, {\em Equations of tensor eigenschemes}, preprint arXiv:2205.04413 [math.AG]
  (2022), {https://arxiv.org/abs/2205.04413}.
 
  	 
\bibitem {CGKS}
  {T. \"{O}. {\c{C}}elik, F. Galuppi, A. Kulkarni, M.-\c{S}. Sorea
              },
 { \em On the eigenpoints of cubic surfaces},
    {Le Matematiche},
    {\bf 75}, n. 2, (2020),
     {611--625}.

\bibitem{CS}
D. Cartwright, B. Sturmfels,
      {\em The number of eigenvalues of a tensor},
    {Linear Algebra and its Applications},
     {\bf 438}, n. 2
      (2013),
    {942--952}.

\bibitem{E} D. Eisenbud, {\em Commutative Algebra.With a view toward algebraic geometry}, Springer-Verlag, Graduate
Texts in Mathematics {\bf 150} (1995).


     
\bibitem{G} E. Gorla, {\em A generalized Gaeta?s theorem},
Compositio Mathematica {\bf 144} (2008), 689 - 704.


\bibitem{HE}
    M. Hochster, John A. Eagon, {\em Cohen-{M}acaulay rings, invariant theory, and the generic
              perfection of determinantal loci},
  {Amer. J. Math.},
  {\bf 93},
      (1971),
     {1020--1058}.
   
   
\bibitem {KMR} J. O. Kleppe, R. M. Mir\'{o}-Roig, 
    {\em The representation type of determinantal varieties},
   Algebr. Represent. Theory, {\bf 20, 4},
      (2017), 1029 - 1059.
    
    \bibitem {Lim}
    L. H. Lim,{\em Singular values and eigenvalues of tensors: a variational approach}, {1st IEEE International Workshop on Computational Advances in Multi-Sensor Adaptive Processing}, 
  (2005),
  129--132.
  
 \bibitem {L} A. F. Lopez, {\em On the curves lying on a general surface containing a fixed
              space curve},
Ricerche Mat.,
 {\bf 41, 1},
(1992),21-40.    
		
		
\bibitem{PT} J. Pons-llopis, F. Tonini, {\em ACM bundles on Del Pezzo surfaces},
Le Matematiche {\bf  LXIV} (2009), 177?211.


\bibitem {Qi}
L. Qi, {\em Eigenvalues of a real supersymmetric tensor}, {Journal of Symbolic Computation},
    {\bf 40}, n. 6
      (2005),
    1302--1324.

\bibitem {QCC}
    L. Qi, H. Chen, Y. Chen,
     {\em Tensor eigenvalues and their applications},
   {Advances in Mechanics and Mathematics},
   {\bf 39} (2018),
 {Springer, Singapore}.
    
      
      
      \bibitem{RS}
     {K. Ranestad, B. Sturmfels},
     {\em Twenty-seven questions about the cubic surface},
   {Le Matematiche},
   {\bf 75}, n. 2, (2020)
     {411--424}.

\bibitem{W}
 Ch. Weibel, {\em An introduction to homological algebra}, Cambridge Studies in Advanced Mathematics {\bf 38} (1994) Cambridge University Press.
\end{thebibliography}
\end{document}